\documentclass[12]{amsart}
\usepackage{amsmath,amssymb,amsthm,color,enumerate,comment,centernot,enumitem,url,cite}
\usepackage{graphicx,relsize,bm}
\usepackage{mathtools}
\usepackage{array}

\usepackage{hyperref}

\makeatletter
\newcommand{\tpmod}[1]{{\@displayfalse\pmod{#1}}}
\makeatother

\newtheorem{thm}{Theorem}[section]
\newtheorem{lemma}[thm]{Lemma}

\newtheorem{cor}[thm]{Corollary}

\theoremstyle{remark}

\theoremstyle{definition}

\theoremstyle{THM}

\newcommand{\W}{{\mathcal W}}

\newcommand{\Gal}{{\mbox{{\rm{Gal}}}}}

\newcommand{\Z}{{\mathbb Z}}
\newcommand{\Q}{{\mathbb Q}}

\newcommand{\F}{{\mathbb F}}

\makeatletter
\@namedef{subjclassname@2020}{%
  \textup{2020} Mathematics Subject Classification}
\makeatother

\def\red#1 {\textcolor{red}{#1 }}
\def\blue#1 {\textcolor{blue}{#1 }}

\numberwithin{equation}{section}

\begin{document}

\title[On the monogenicity and Galois groups of $\boldsymbol{x^{2p}+ax^p+b^p}$]{On the monogenicity and Galois groups of $\boldsymbol{x^{2p}+ax^p+b^p}$}

\author{Joshua Harrington}
\address{Department of Mathematics, Cedar Crest College, Allentown, Pennsylvania, USA}
\email[Joshua Harrington]{Joshua.Harrington@cedarcrest.edu}

\author{Lenny Jones}
\address{Professor Emeritus, Department of Mathematics, Shippensburg University, Shippensburg, Pennsylvania 17257, USA}
\email[Lenny~Jones]{doctorlennyjones@gmail.com}

\date{\today}

\begin{abstract}
    Let $f(x)=x^{2p}+ax^p+b^p$, where $p$ is a prime and $a,b\in \Z$ with $ab\ne 0$. If $f(x)$ is irreducible over $\Q$, we say that $f(x)$ is {\em monogenic} if $\{1,\theta,\theta^2,\ldots ,\theta^{2p-1}\}$
    is a basis for the ring of integers of ${\mathbb Q}(\theta)$, where $f(\theta)=0$. 
    
    In this article, we give a characterization of the monogenic trinomials $f(x)$ according to their Galois groups. These results extend prior investigations of the authors.
    
\end{abstract}


\maketitle
\section{Introduction}\label{Section:Intro}
 Throughout this article, we let   
 \begin{equation}\label{Eq:f}
 \begin{gathered}
 f(x)=x^{2p}+ax^p+b^p \quad \mbox{and} \quad \delta=a^2-4b^p,\\
 \mbox{where $p\ge 3$ is prime and $a,b\in \Z$ with $ab\ne 0$.}
 \end{gathered}
 \end{equation} 

  We say $f(x)$ is {\em monogenic} if $f(x)$ is irreducible over $\Q$ and $\{1,\theta,\theta^2,\ldots,\theta^{2p-1}\}$ is a basis for the ring of integers of $K={\mathbb Q}(\theta)$, denoted $\Z_K$, where $f(\theta)=0$. It is well known for $f(x)$ irreducible over $\Q$ that \cite{Cohen}
\begin{equation} \label{Eq:Dis-Dis}
\Delta(f)=\left[\Z_K:\Z[\theta]\right]^2\Delta(K),
\end{equation} 
where $\left[\Z_K:\Z[\theta]\right]$ is called the {\em index} of $f(x)$, while $\Delta(f)$ and $\Delta(K)$ denote the respective discriminants over $\Q$ of $f(x)$ and the number field $K$.
Thus, 
$f(x)$ is monogenic if and only if $\Delta(f)=\Delta(K)$, or equivalently, $\Z_K=\Z[\theta]$ from \eqref{Eq:Dis-Dis}.
 
With $f(x)$ irreducible over $\Q$, the authors determined $\Gal(f)$ for $p=3$ in \cite[Theorem 4.2]{HJMS}. The second author extended this result in \cite{JonesJAA} to the situation when $p\ge 5$ with $b=1$. Recently, Jim and Hagedorn have determined $\Gal(f)$ for $p\ge 5$ without the restriction that $b=1$ \cite{JH}. Combining all these results yields 
\begin{thm}\label{Thm:Galois} Suppose that $f(x)$, as defined in \eqref{Eq:f}, is irreducible over $\Q$, and let $C_N$ denote the cyclic group of order $N$. Then
\begin{equation}
\Gal(f)\simeq\left\{\begin{array}{cl}
 C_p\rtimes C_{p-1} & \mbox{if $p\equiv 1 \tpmod{4}$ and}\\
  &  \mbox{$\delta=py^2$ for some $y\in \Z$}\\ 
 \left(C_p\rtimes C_{(p-1)/2}\right)\times C_2 & \mbox{if $p\equiv 3 \tpmod{4}$ and}\\
  & \mbox{$\delta=-py^2$ for some $y\in \Z$}\\ 
 \left(C_p\rtimes C_{p-1}\right)\times C_2 & \mbox{if $p\ge 3$ and}\\
 & \mbox{$\delta \ne (-1)^{p(p-1)/2}py^2$ for any $y\in \Z$.} 
\end{array}\right.
\end{equation}
\end{thm} 

If $p\nmid \delta$, we see from Theorem \ref{Thm:Galois} that $\Gal(f)\simeq \left(C_p\rtimes C_{p-1}\right)\times C_2$. Since, either all or none of the trinomials in $\{\pm f(\pm x),\pm \tilde{f}(\pm x)\}$ are monogenic, where $\tilde{f}$ denotes the reciprocal of $f$, the characterization of the monogenic trinomials $f(x)$ with $\Gal(f)\simeq \left(C_p\rtimes C_{p-1}\right)\times C_2$ follows from \cite{JonesTJM,JonesAJM}. Consequently, in this article, we focus on the situation when $p\mid \delta$, and    
we provide a complete characterization of the monogenic trinomials $f(x)$, according to their Galois groups in Theorem \ref{Thm:Galois}, in this situation.  
A rather striking consequence of this investigation is highlighted below in Corollary \ref{Cor:Main}, which links the monogenicity of a polynomial with the existence of primes of a certain form.   
This work extends prior investigations of the authors \cite{HJBAMS,JonesAJM,HJActa,HJJAA,JonesJAA,JonesTJM}. 
Our main results are  
 \begin{thm}\label{Thm:Main} 
 Let $p$, $f(x)$ and $\delta$ be as given in \eqref{Eq:f}. Suppose that $f(x)$ is irreducible over $\Q$, and that $p\mid \delta$. 
 \begin{enumerate}[label=(\arabic*)] 
  \item \label{I1} If $\Gal(f)\simeq C_p\rtimes C_{p-1}$, then 
   \[f(x) \ \mbox{is monogenic} \ \Longleftrightarrow \  
    (p,a,b)\in \{(5,\pm 3,1), (a^2+4,a,-1)\}.\] 
    \item \label{I2} If $\Gal(f)\simeq \left(C_p\rtimes C_{(p-1)/2}\right)\times C_2$, then  
 \[f(x) \ \mbox{is monogenic} \ \Longleftrightarrow \ (p,a,b)\in \{(3,\pm 1,1)\}.\] 
 \item \label{I3} If $\Gal(f)\simeq \left(C_p\rtimes C_{p-1}\right)\times C_2$, then 
 \begin{enumerate}
   \item if $\delta=(-1)^{p(p+1)/2}py^2$ for some $y$, then 
   \[f(x) \ \mbox{is monogenic} \ \Longleftrightarrow \ (p,a,b)\in \{(3,\pm4,1)\},\] 
   \item \label{last} if $\delta\ne (-1)^{p(p+1)/2}py^2$ for any $y$, then 
  \[f(x) \ \mbox{is monogenic} \ \Longleftrightarrow \ \left\{\begin{array}{l}
    b=1 \ \mbox{with $a-2$ and $a+2$ squarefree,}\\
    \qquad \qquad \mbox{or}\\
    b=-1 \ \mbox{with $4\nmid a$ and $(a^2+4)/\gcd(2,a)^2$ squarefree.}
  \end{array}\right.\]
  Moreover, for any prime $p\ge 3$, there exist infinitely many such trinomials $f(x)$ for each case of item \ref{last}.  
  \end{enumerate} 
  \end{enumerate} 
  \end{thm}  
  The following interesting corollary of Theorem \ref{Thm:Main} emphasizes the 
  connection between the existence of infinitely many primes of a particular quadratic form and the existence of infinitely many monogenic trinomials $f(x)$ of a certain form with $\Gal(f)\simeq C_p\rtimes C_{p-1}$. 
  \begin{cor}\label{Cor:Main} Let $a$ and $z$ denote positive integers. For any integer $N>2$, there exists a prime $p>N$ such that $f(x)=x^{2p}+ax^p-1$ is monogenic with $\Gal(f)\simeq C_p\rtimes C_{p-1}$ for some $a$ (and only one such $a$ exists) if and only if there exist infinitely many primes of the form $z^2+4$. 
  \end{cor}

\section{Preliminaries}\label{Section:Prelim}
The following theorem is due to Kaur, Kumar and Remete \cite{KKR}. 
 \begin{thm}\label{Thm:KKR}
   Let $g(x)$ be a monic polynomial with integer coefficients. Let $k\ge 2$ be an integer such that ${\mathfrak f}(x):=g(x^k)$ is irreducible over $\Q$. 
   Then ${\mathfrak f}(x)$ is monogenic if and only if all of the following conditions are satisfied: 
   \begin{enumerate}
   \item \label{KKR:I1} $g(0)$ is squarefree,
    \item \label{KKR:I2} $q$ does not divide the index of ${\mathfrak f}(x)$ for all primes $q\mid k$,
     \item \label{KKR:I3} $g(x)$ is monogenic. 
   \end{enumerate}
 \end{thm}

The next theorem, due to Jakhar, Khanduja and Sangwan \cite{JKS2}, gives necessary and sufficient conditions to determine when a prime divisor of $\Delta({\mathfrak F})$ does not divide the index of ${\mathfrak F}$, where ${\mathfrak F}$ is an arbitrary irreducible monic trinomial. 
\begin{thm}\label{Thm:JKS2} 
Let $n,m\in \Z$ with $n>m\ge 1$.
Let $K=\Q(\theta)$ be an algebraic number field with $\theta\in \Z_K$, the ring of integers of $K$, having minimal polynomial ${\mathfrak F}(x)=x^{n}+Ax^m+B$ over $\Q$, where $\gcd(m,n)=d_0$, $m=m_1d_0$ and $n=n_1d_0$. A prime factor $q$ of $\Delta({\mathfrak F})$ does not divide $\left[\Z_K:\Z[\theta]\right]$ if and only if $q$ satisfies one of the following conditions:
 \begin{enumerate}[label=(\roman*), font=\normalfont]
  \item \label{JKS:C1} when $q\mid A$ and $q\mid B$, then $q^2\nmid B$;
  \item \label{JKS:C2} when $q\mid A$ and $q\nmid B$, then
  \[\mbox{either } \quad q\mid a_2 \mbox{ and } q\nmid b_1 \quad \mbox{ or } \quad q\nmid a_2\left((-B)^{m_1}a_2^{n_1}+\left(-b_1\right)^{n_1}\right),\]
  where $a_2=A/q$ and $b_1=\frac{B+(-B)^{q^j}}{q}$, such that $q^j\mid\mid n$ with $j\ge 1$;
  \item \label{JKS:C3} when $q\nmid A$ and $q\mid B$, then
  \[\qquad \quad \mbox{either}\quad q\mid a_1 \mbox{ and } q\nmid b_2  \quad\mbox{or}\quad  q\nmid a_1b_2^{m-1}\left((-A)^{m_1}a_1^{n_1-m_1}-\left(-b_2\right)^{n_1-m_1}\right),\]
  where $a_1=\frac{A+(-A)^{q^l}}{q}$, such that $q^l\mid\mid (n-m)$ with $l\ge 0$, and $b_2=B/q$;
  \item \label{JKS:C4} when $q\nmid AB$ and $q\mid m$ with $n=s^{\prime}q^r$, $m=sq^r$, $q\nmid \gcd\left(s^{\prime},s\right)$, then the polynomials
   \begin{equation*}
     H_1(x):=x^{s^{\prime}}+Ax^s+B \quad \mbox{and}\quad H_2(x):=\dfrac{Ax^{sq^r}+B+\left(-Ax^s-B\right)^{q^r}}{q}
   \end{equation*}
   are coprime modulo $q$;
         \item \label{JKS:C5} when $q\nmid ABm$, then
     \[q^2\nmid \left(B^{n_1-m_1}n_1^{n_1}-(-1)^{m_1}A^{n_1}m_1^{m_1}(m_1-n_1)^{n_1-m_1}\right).\]
   \end{enumerate}
\end{thm}
In light of Theorem \ref{Thm:KKR}, we use Theorem \ref{Thm:JKS2} to prove the following useful lemma. 
\begin{lemma}\label{Lem:quadratic}
 Let $g(x)=x^2+ax+b$, where $a$ is a nonzero integer and $b\in \{\pm 1\}$. Suppose that $g(x)$ is irreducible over $\Q$. Let $\W:=\delta/\gcd(a,2)^2$. Then  
  \[g(x) \ \mbox{is monogenic} \ \Longleftrightarrow \  \W \ \mbox{is squarefree, with} \ \left\{\begin{array}{l}
   \mbox{$4\mid a$ if $2\mid a$, when $b=1$,}\\
   \mbox{$4\nmid a$, when $b=-1$.}
  \end{array}\right.\] 
\end{lemma}
\begin{proof} Observe that $\Delta(g)=\delta$. We examine Theorem \ref{Thm:JKS2} with $A=a$, $B=b$ and $q$ a prime divisor of $\delta$. 

Suppose first that $b=1$. If $q\mid a$, then $q=2$, so that $b_1=1$ in condition \ref{JKS:C2} of Theorem \ref{Thm:JKS2}. It follows that $2\nmid [\Z_K:\Z[\theta]]$ if and only if $4\mid a$. If $q\nmid a$, then $q\nmid [\Z_K:\Z[\theta]]$ if and only if $q^2\nmid \W$ is squarefree by condition \ref{JKS:C5} of Theorem \ref{Thm:JKS2}, which completes the proof when $b=1$.

Suppose next that $b=-1$. 
If $q\mid a$, then $q=2$, so $b_1=0$ in condition \ref{JKS:C2} of Theorem \ref{Thm:JKS2}. Thus, $2\nmid [\Z_K:\Z[\theta]]$ if and only if $4\nmid a$. Suppose that $q\nmid a$. Then, by condition \ref{JKS:C5} of Theorem \ref{Thm:JKS2}, we see that $q\nmid [\Z_K:\Z[\theta]]$ if and only if $\W$ is squarefree. 
\end{proof}

\section{The Proof of Theorem \ref{Thm:Main}}\label{Section:MainThmProof}
\begin{proof}[Proof of Theorem \ref{Thm:Main}] 
We first examine the three conditions of Theorem \ref{Thm:KKR} to determine criteria for the monogenicity of $f(x)$. 
In the context of Theorem \ref{Thm:KKR}, with ${\mathfrak f}(x)$ equal to our trinomial $f(x)$, we have that $f(x)=g(x^p)$, where $g(x)=x^{2}+ax+b^p$. Note that $p\mid \Delta(f)$ since $\Delta(f)=b^{p(p-1)}p^{2p}(a^2-4b^p)^p$, by a theorem due to Swan \cite{Swan}.

Condition \eqref{KKR:I1} of Theorem \ref{Thm:KKR} requires that $g(0)$ be squarefree, which implies that $b=\pm 1$. 
For a prime $p\ge 3$, with $j\in \{1,2\}$ and $b=\pm 1$, let 
\begin{align}\label{Char}
\begin{split}
 (\Gamma_j,b) \ \ & \mbox{denote the condition that} \ \delta=(-1)^{j-1}py^2 \ \mbox{for some $y\in \Z$, and}\\
 (\Omega,b) \ \ & \mbox{denote the condition that} \ \delta\ne \pm py^2 \ \mbox{for any $y\in \Z$.}
 \end{split}
 \end{align}

Condition \eqref{KKR:I2} of Theorem \ref{Thm:KKR} indicates that we only have to check the single prime $q=p$ in Theorem \ref{Thm:JKS2} to determine conditions under which the index of $f(x)$ is not divisible by a prime divisor $q$ of $\Delta(f)$. Since $p\mid \delta$, where $p$ is odd, and $b=\pm 1$, we deduce that we only have to check condition \ref{JKS:C4} of Theorem \ref{Thm:JKS2}. 

Condition \eqref{KKR:I3} of Theorem \ref{Thm:KKR} requires that $g(x)$ be monogenic for the monogenicity of $f(x)$. 
For the four cases $(\Gamma_j,b)$ described in \eqref{Char}, it follows from 
  Lemma \ref{Lem:quadratic} that 
 \begin{equation}\label{Squarefree}
 \begin{gathered}
 \W=\frac{a^2-4b}{\gcd(a,2)^2}=\frac{(-1)^{j-1}py^2}{\gcd(a,2)^2} \quad \mbox{is squarefree, so that}\\
 a^2-4b\in \{(-1)^{j-1}p,(-1)^{j-1}4p\},
 \end{gathered}
 \end{equation} 
 where $4\mid a$ if $2\mid a$, when $b=1$, and $4\nmid a$ when $b=-1$. Clearly, of the four cases $(\Gamma_j,b)$, the case $(\Gamma_2,-1)$ is not possible.  Consequently, we examine the three cases $(\Gamma_j,b)\in \{(\Gamma_1,1), (\Gamma_2,1), (\Gamma_1,-1)\}$, and the two cases $(\Omega,\pm 1)$.

\subsection*{{\bf The Case} $\mathbf{(\Gamma_1,1)}$} From \eqref{Char} and \eqref{Squarefree}, we have that $a^2-4\in \{p,4p\}$.

If $a^2-4=p$, then either $a-2=1$ and $a+2=p$ or $a-2=-p$ and $a+2=-1$. From the first possibility we get $(p,a)=(5,3)$, while the second possibility produces the pair $(p,a)=(5,-3)$. These pairs correspond to the trinomials 
\begin{equation}\label{Tris1}
f(x)=x^{10}\pm 3x^5+1.
\end{equation} Examining condition \ref{JKS:C4} of Theorem \ref{Thm:JKS2} with $q=5$, we see that 
\[H_1(x)\equiv (x\pm 4)^2 \tpmod{5} \quad \mbox{and} \quad H_2(x)\equiv \pm 2x(x^4\pm 2x^3+3x^2\pm x+1)\tpmod{5},\] so that $H_1(x)$ and $H_2(x)$ are coprime in $\F_5[x]$.  Thus, the two trinomials in \eqref{Tris1} are monogenic, and we deduce from Theorem \ref{Thm:Galois} that $\Gal(f)\simeq C_5\rtimes C_4$ for both, since $p\equiv 1 \pmod{4}$. 

 If $a^2-4=4p$, then $4\mid a$ so that $(a/2)^2-1=p$. Arguing as before, we get that $(p,a)\in \{(3,\pm 4)\}$. These pairs yield the monogenic trinomials $f(x)=x^6\pm 4x^3+1$, with $\Gal(f)\simeq \left(C_3\rtimes C_2\right)\times C_2$ from Theorem \ref{Thm:Galois}, since $p\equiv 3 \pmod{4}$. 
 
 Of course, all of the above conclusions can be easily verified using a computer algebra system.  
 
 \subsection*{{\bf The Case} $\mathbf{(\Gamma_2,1)}$} From \eqref{Char} and \eqref{Squarefree}, we have that $a^2-4\in \{-p,-4p\}$.
The only viable situation here is $a^2-4=-p$, which implies that $p=3$ and $a=\pm 1$. Hence, $f(x)=x^6\pm x^3+1$. From condition \ref{JKS:C4} of Theorem \ref{Thm:JKS2} with $q=p$, we get
\[H_1(x)\equiv (x\pm 2)^2 \tpmod{3}\quad \mbox{and} \quad H_2(x)=-x(x\pm 1) \tpmod{3},\] which verifies that $H_1(x)$ and $H_2(x)$ are coprime in $\F_3[x]$. Thus, the trinomials $f(x)$ are monogenic, and 
$\Gal(f)\simeq \left(C_3\rtimes C_{1}\right)\times C_2\simeq C_6$ by Theorem \ref{Thm:Galois} since $p\equiv 3 \pmod{4}$.  

\subsection*{{\bf The Case} $\mathbf{(\Gamma_1,-1)}$} From \eqref{Char} and \eqref{Squarefree}, we have that $a^2+4\in \{p,4p\}$. 

Notice that if $a^2+4=4p$, then $p=(a/2)^2+1$, which implies that $p=2$ since $4\nmid a$, contradicting the fact that $p\ge 3$. Thus, $a^2+4=p$ and $p\equiv  1 \pmod{4}$. For monogenicity, we examine condition \ref{JKS:C4} of Theorem \ref{Thm:JKS2} with $q=p$. Then
 \[H_1(x)=x^2+ax-1\equiv (x+a/2)^2 \pmod{p},\] and  
 \begin{align*}
    H_2(-a/2)&=\frac{a(-a/2)^p-1+(-a(-a/2)+1)^p}{p}\\
   &=\frac{-(a^2)^{(p+1)/2}-2^p+\left(a^2+2\right)^p}{p2^p}\\
   &=\frac{-(p-4)^{(p+1)/2}-2^p+\left(p-2\right)^p}{p2^p}\\
   &=\frac{-\sum_{j=0}^{(p+1)/2}\binom{(p+1)/2}{j}p^j(-4)^{(p+1)/2-j}-2^p+\sum_{j=0}^p\binom{p}{j}p^j(-2)^{p-j}}{p2^p}\\
   &\equiv \frac{-(-4)^{(p+1)/2}-\frac{p(p+1)}{2}(-4)^{(p-1)/2}-2^p+(-2)^p}{p2^p} \pmod{p}\\
   &=\frac{(-1)^{(p+3)/2}2^{p+1}-\frac{p(p+1)}{2}(-1)^{(p-1)/2}2^{p-1}-2^{p+1}}{p2^p}\\
   &=-\left(\frac{p+1}{4}\right) \not \equiv 0 \pmod{p}.
 \end{align*}
 Hence, $H_1(x)$ and $H_2(x)$ are coprime in $\F_p[x]$. Therefore, $f(x)$ is monogenic with $\Gal(f)\simeq C_p\rtimes C_{p-1}$ from Theorem \ref{Thm:Galois}. 
 
 \subsection*{{\bf The Case} $\mathbf{(\Omega,1)}$} From Lemma \ref{Lem:quadratic}, we require that 
  \begin{equation}\label{SFconditions}
  \W=\delta/\gcd(a,2)^2=\left\{\begin{array}{cl}
    (a-2)(a+2) \ \mbox{is squarefree} & \mbox{if $2\nmid a$}\\[.5em]
    (a/2-1)(a/2+1) \ \mbox{is squarefree} & \mbox{if $2\mid a$,}  
  \end{array}\right.
   \end{equation} for the monogenicity of $f(x)$. It is easy to see that the squarefree conditions in \eqref{SFconditions} are equivalent to each of the factors of $\W$ being squarefree, regardless of the parity of $a$. Furthermore, since $p\ge 3$ and $p\mid \delta$, we see that $p\mid (a-2)$ or $p\mid (a+2)$, whether or not $2\mid a$. 
   
  We give details only for $p\mid(a-2)$ since $p\mid (a+2)$ is similar.
   Writing $a=pk+2$ for some $k\in \Z$, and examining condition \ref{JKS:C4} of Theorem \ref{Thm:JKS2}, we have 
   \[H_1(x)=x^{2}+(pk+2)x+1\equiv (x+1)^2 \pmod{p},\] and  
   \begin{align*}
     H_2(-1)&=\frac{-a+1+(a-1)^p}{p}\\  
     &=\frac{-pk-1+(pk+1)^p}{p}\\
     &=\frac{-pk-1+\sum_{j=0}^p\binom{p}{j}(pk)^j}{p}\\
     &=\frac{-pk+\sum_{j=1}^p\binom{p}{j}(pk)^j}{p}\\
     &\equiv -k \pmod{p}\\
     &\not\equiv 0 \pmod{p},
        \end{align*} since $a-2=pk$ is squarefree. Hence, $H_1(x)$ and $H_2(x)$ are coprime in $\F_p[x]$, so that $f(x)$ is monogenic. We see from Theorem \ref{Thm:Galois} that $\Gal(f)\simeq \left(C_p\rtimes C_{p-1}\right)\times C_2$. Since there exist infinitely many integers $a$ such that $(a-2)(a+2)$ is squarefree when $2\nmid a$, and infinitely many integers $a$ such that $(a/2-1)(a/2+1)$ is squarefree when $2\mid a$ \cite{BB}, it follows that there exist infinitely many $a\in \Z$ such that $f(x)$ is monogenic with $\Gal(f)\simeq \left(C_p\rtimes C_{p-1}\right)\times C_2$.  

\subsection*{{\bf The Case} $\mathbf{(\Omega,-1)}$} Since $p\mid \delta$ and $f(-x)$ is monogenic if and only if $f(x)$ is monogenic, this case follows from \cite[Theorem 1.1]{JonesTJM}. There exist infinitely many such trinomials since there exist infinitely many integers $a$ such that $(a^2+4)/\gcd(a,2)^2$ is squarefree \cite{Nagel}.  
\end{proof}

\section{The Proof of Corollary \ref{Cor:Main}}\label{Section:MainCorProof}
\begin{proof}
Suppose that $p$ is a prime with $p>N$ such that, for some integer $a>0$,  
 $f(x)=x^{2p}+ax^p-1$ is monogenic with $\Gal(f)\simeq C_p\rtimes C_{p-1}$. Then $p=a^2+4$ from Theorem \ref{Thm:Main}, and clearly, $a$ is unique. Conversely, suppose that $z^2+4$ is prime and let $f(x)=x^{2(z^2+4)}+zx^{z^2+4}-1=g(x^{z^2+4})$, where $g(x)=x^2+zx-1$. Since $g(x)$ is irreducible over $\Q$ and $f(x)$ is irreducible over $\Q$ if and only if $f(-x)$ is irreducible over $\Q$, it follows from \cite[Lemma 3.1]{JonesAJM} that $f(x)$ is irreducible over $\Q$. Then, we have from Theorem \ref{Thm:Main} that $f(x)$ is monogenic with $\Gal(f)\simeq C_p\rtimes C_{p-1}$, since $z^2+4\equiv 1 \pmod{4}$.     
\end{proof}


\section*{Acknowledgments} 







\end{document}